 \documentclass[gmj]{degruyter-journal-a}      

\title{Direct and inverse theorems of approximation of functions in
weighted Orlicz type spaces \\ with variable exponent}
\headlinetitle{Direct and inverse theorems of approximation...}

\lastnameone{Abdullayev}
\firstnameone{Fahreddin  G.}
\nameshortone{F.\,G.~Abdullayev}
\addressone{Kygyz-Turkish Manas University,  Chyngyz Aitmatov avenue 56, 720044, Kyrgyz Republic;
Mersin University, \c{C}iftlikk\"{o}y Kamp\"{u}s\"{u}, Yeni\c{s}ehir, Mersin 33343 }
\countryone{Turkey}
\emailone{fahreddinabdullayev@gmail.com}

\lastnametwo{Chaichenko}
\firstnametwo{Stanislav O.}
\nameshorttwo{S.\,O.~Chaichenko}
\addresstwo{Donbas State Pedagogical University,  G.~Batyuka st. 19, Slaviansk, Donetsk region 84100}
\countrytwo{Ukraine}
\emailtwo{s.chaichenko@gmail.com}

\lastnamethree{Imash kyzy}
\firstnamethree{Meerim}
\nameshortthree{M.~Imash kyzy}
\addressthree{Kygyz-Turkish Manas University,  Chyngyz Aitmatov avenue 56, 720044}
\countrythree{Kyrgyz Republic}
\emailthree{imashkyzy@gmail.com}

\lastnamefour{Shidlich}
\firstnamefour{Andrii L.}
\nameshortfour{A.\,L.~Shidlich}
\addressfour{Institute of Mathematics of NAS of Ukraine,  Tereshchenkivska str. 3, Kyiv 01024}
\countryfour{Ukraine}
\emailfour{shidlich@gmail.com}

\abstract{In weighted Orlicz type  spaces ${\mathcal S}_{_{\scriptstyle  \mathbf p,\,\mu}}$
with a variable summation  exponent,  the direct and inverse approximation
theorems are proved in terms of 
best approximations of functions and moduli of smoothness of fractional order. It is shown that the constant obtained in the inverse approximation theorem is in a certain sense the best. Some applications of the results are also proposed. In particular,  the constructive characteristics of functional classes defined by such moduli of smoothness  are given. Equivalence between moduli of smoothness and certain Peetre $K$-functionals is 
shown in the spaces
${\mathcal S}_{_{\scriptstyle  \mathbf p,\,\mu}}$.}



\researchsupported{This work was supported in part by the Ministry of Education and Science of
 Ukraine within the framework of the fundamental research No.~0118U003390 and by
 the Kyrgyz-Turkish Manas University (Bishkek / Kyrgyz Republic),
 project No.~KTM\"{U}-BAP-2018.FBE.05.}

\begin{document}

\section{Introduction} Let $C^r({\mathbb T})$ (${\mathbb T}:=[0,2\pi]$, $r\in {\mathbb N}_0:=\{0,1,\ldots\}$)
denote the space of $2\pi$-periodic $r$-times continuously differentiable functions with the usual  max-norm $\|f\|=\max_{x\in {\mathbb  T}} |f(x)|$. Let also $E_n(f)=\inf\limits_{\tau_n}\|f-\tau_n\|$ be
 the best approximation  of function $f\in C({\mathbb T})$ by trigonometric polynomials $\tau_n$ of degree $n$, $n\in {\mathbb N}_0$. The classical theorem of Jackson  (1912) says that {\it i) if  $f\in C^r({\mathbb T})$, then the following inequality holds:
 $
 E_n(f)\le K_r n^{-r}\omega (f^{(r)},n^{-1})$, $n=1,2,\ldots,$
 where  $ \omega (f,t):= \sup\limits_{|h|\le t}\|f(\cdot+h)-f(\cdot)\|$
 is the modulus of continuity of $f$.} This assertion   is a {\it direct approximation theorem}, which asserts that
smoothness of the function $f$ implies a quick decrease to zero of its error of
approximation by trigonometric polynomials.

On the other hand, the following {\it inverse theorem} of  Bernstein   (1912) with the opposite
implication  is well-known: {\it ii) if for some $0\!<\alpha\!<1$,
$E_n(f)\!\le K_r n^{-r-\alpha}$, $n=1,2,\ldots$,
then 
$\omega (f^{(r)},t)={\mathcal O}(t^\alpha)$, $t\to 0+$.} In ideal cases, these two theorems
correspond to each other. For example,
it follows from i) and ii) that  the relation $E_n(f)={\mathcal O}(n^{-\alpha})$, $0<\alpha<1$, is equivalent
to the condition $\omega (f,t)={\mathcal O}(t^\alpha)$, $t\to 0+$.
Such theorems have been of great interest to researchers and constitute the classics of modern approximation  theory (see, for example the monographs \cite{Akhiezer_M1947}, \cite{Butzer_Nessel_M1971},
\cite{A_Timan_M1960}, \cite{DeVore_Lorentz_M1993}, \cite{Dzyadyk_Shevchuk_M2008},
\cite{M_Timan_M2009}).



In recent decades,   the topics related to  the direct and inverse approximation theorems  have been
actively investigated in  the Orlicz spaces and in the Lebesgue spaces with a variable  exponent. In particular, for the Lebesgue functional spaces with variable exponent, similar results are contained in the papers of Guven and Israfilov \cite{Guven-Israfilov-2010}, Akg\"{u}n \cite{Akgun-umg.2011.1}, Akg\"{u}n and Kokilashvili \cite{Akgun-Kokilashvili-gmg-2011-3, Akgun-Kokilashvili-Jour-math-scien-2012-4},  Chaichenko \cite{Chaichenko_umg_2012_engl}, Jafarov \cite{Jafarov-2017-DM, Jafarov-2018-JCVEE} and others.
The latest results related to  the Lebesgue spaces with variable exponent, and their applications are described in the monograph \cite{Diening_and_other_book_2011}. We also note
the papers by Nekvinda \cite{Nekvinda-JFSA-1-2007, Nekvinda-MIA-1-2007} devoted to the investigations of  the discrete weighted Lebesgue spaces with a variable exponent.


In 2000, Stepanets \cite{Stepanets_2001} considered the spaces  $S^p$ of $2\pi$-periodic  Lebesgue summable functions $f$ ($f\in L$) with the finite norm
\[
    \|f\|_{_{\scriptstyle S^p}}:=\|\{\widehat f({k})\}_{{k}\in\mathbb  Z}
    \|_{_{\scriptstyle l_p({\mathbb Z})}}=\Big(\sum_{{ k}\in\mathbb  Z}|\widehat f({k})|^p\Big)^{1/p},
\]
where
$\widehat{f}(k):={[f]}\widehat{\ \ }(k)=(2\pi)^{-1}\int_0^{2\pi}f(t) \mathrm{e}^{- \mathrm{i}kt}{\rm d}t$,
$k\in\mathbb Z$, are the Fourier coefficients of the function $f$,
and  investigated some approximation characteristics of these spaces,
including in the context of direct and inverse theorems.
Stepanets and Serdyuk \cite{Stepanets_Serdyuk_2002} introduced the notion of $k$th modulus of smoothness in $S^p$ and established   the  direct and inverse theorems on approximation in terms of these moduli of
smoothness  and the best approximations of  functions. Also this topic was
investigated actively in \cite{Sterlin_1972}, \cite{Serdyuk_2003}, \cite{Vakarchuk_2004},  \cite[Ch.~9]{Stepanets_M2005},  \cite[Ch.~3]{M_Timan_M2009}  and others.

In the papers  \cite{Shydlich_Chaichenko_2004}, \cite{Shidlich_Chaichenko_2005} some results for the spaces $S^p$ were extended to the Orlicz spaces  $l_{M}$ and to the spaces $l_{\mathbf p}$ with a variable summation  exponent. In particular, in these spaces,  the authors found  the exact values of the best approximations
and Kolmogorov's widths of certain sets of images of the diagonal operators. The purpose of this paper is to combine the above mentioned studies and  prove  the  direct and inverse approximation theorems  in the weighted spaces ${\mathcal S}_{_{\scriptstyle  \mathbf p,\,\mu}}$  of the Orlicz type with a variable summation  exponent.


\section{ Preliminaries.}  Let ${\mathbf p}=\{p_k\}_{k=-\infty}^\infty$ be a sequence of positive numbers such that
\begin{equation}\label{l_p.2001}
1\le p_k\le K,\quad k=0,\pm 1, \pm 2,\ldots,
\end{equation}
where $K$ is a positive number, and ${\mathbf \mu}=\{\mu_k\}_{k=-\infty}^\infty$ be a sequence of nonnegative  numbers. Let ${\mathcal S}_{_{\scriptstyle  \mathbf p,\,\mu}}$ be the space of all functions $f\in L$ such that the following quantity (which is also called the Luxemburg norm of $f$) is finite:
\begin{equation}\label{S_M.1}
\|{f}\|_{_{\scriptstyle \mathbf p,\,\mu}}:=
    \|\{\widehat{f}(k)\}_{k\in {\mathbb Z}}\|_{_{\scriptstyle l_{\mathbf p,\,\mu}({\mathbb Z})}}
    =
    \inf\bigg\{a>0:\  \sum\limits_{k\in\mathbb Z}   \mu_k|{\widehat{f}(k)}/{a} |^{p_k}\le 1\bigg\}.
\end{equation}
The functions $f\in L$ and $g\in L$ are equivalent in 
${\mathcal S}_{_{\scriptstyle  \mathbf p,\,\mu}}$, when 
$\|f-g\|_{_{\scriptstyle  \mathbf p,\,\mu}}=0.$

If the sequence ${\mathbf p}=\{p_k\}_{k=-\infty}^\infty$ satisfies condition  (\ref{l_p.2001}), then
\[
{\mathcal S}_{_{\scriptstyle  \mathbf p,\,\mu}}=\bigg\{f\in L\ :\quad \sum\limits_{k\in\mathbb Z}   \mu_k|{\widehat{f}(k)}|^{p_k}<\infty\bigg\}.
\]

The spaces ${\mathcal S}_{_{\scriptstyle  \mathbf p,\,\mu}}$ defined in this way are the  Banach spaces.
In case when $p_k=p$ and $\mu_k= 1$,  ${k\in\mathbb Z}$,  $p\ge 1$, they coincide with the above-defined spaces $S^p$.

Let ${\mathcal T}_{n}$, $n=0,1,\ldots$, be the set of all trigonometric polynomials
$\tau_{n}(x):=\sum_{|k|\le n}  c_{k}\mathrm{e}^{\mathrm{i}kx}$ of the order $n$, where $c_{ k}$
are arbitrary complex numbers. For any function $f\in {\mathcal S}_{_{\scriptstyle  \mathbf p,\,\mu}}$, we denote
by
\[
    E_n (f)_{_{\scriptstyle \mathbf p,\,\mu}}:=
    \inf\limits_{\tau_{n-1}\in {\mathcal T}_{n-1} }\|f-\tau_{n-1}\|_{_{\scriptstyle \mathbf p,\,\mu}}=
    \inf\limits_{c_{ k}\in {\mathbb C}}\Big\|f-\sum_{|k|\le n-1}
    c_{k}\mathrm{e}^{\mathrm{i}k\cdot}\Big\|_{_{\scriptstyle \mathbf p,\,\mu}}
\]
the best approximation of $f$ by the  trigonometric polynomials $\tau_{n-1}\in {\mathcal T}_{n-1}$ in the space ${\mathcal S}_{_{\scriptstyle  \mathbf p,\,\mu}}$.

For a fixed $a>0$ and arbitrary numbers  $c_k\in {\mathbb C}$,
\[%
    \sum\limits_{|k|\le n-1} \mu_k\Big({|\widehat{f}(k)-c_k|}/{a}\Big)^{p_k}+ \sum\limits_{|k|\ge n} \mu_k\Big({|\widehat{f}(k)|}/{a}\Big)^{p_k}\ge \sum\limits_{|k|\ge n} \mu_k\Big({|\widehat{f}(k)|}/{a}\Big)^{p_k},
\]%
therefore, for any function $f\in {\mathcal S}_{_{\scriptstyle  \mathbf p,\,\mu}}$  we have
\begin{eqnarray}\nonumber
E_n (f)_{_{\scriptstyle \mathbf p,\,\mu}}&=&\|f-{S}_{n-1}({f})\|_{_{\scriptstyle \mathbf p,\,\mu}}\\
\label{S_M.4}
&=&\inf\bigg\{a>0: \sum\limits_{|k|\ge n} \mu_k\Big({|\widehat{f}(k)|}/{a}\Big)^{p_k}\le 1\bigg\}.
\end{eqnarray}

where
$S_{n-1}(f,x)= \sum _{|k|\le n-1}\widehat{f}(k) {\mathrm{e}^{\mathrm{i}kx}}$
is the Fourier sum of the function $f$.


\section{ Differences and moduli of smoothness of fractional order.}

Similarly to \cite{Butzer_Westphal_1975},  we define the (right) difference of $f\in L$ of the fractional order  $\alpha>0$
with respect to the increment $h\in {\mathbb R}$ by
\begin{equation}\label{S_M.6}
    \Delta_h^\alpha f({ x}):=\sum\limits_{j=0}^\infty (-1)^j {\alpha \choose j} f({ x}-jh),\ \
\end{equation}
where ${\alpha \choose j}=\frac {\alpha(\alpha-1)\cdot\ldots\cdot(\alpha-j+1)}{j!},~ j \in \mathbb{N}$, ${\alpha \choose 0}:=1$, and assemble some basic properties of the fractional differences.

\begin{lemma}\label{Lemma_1} Assume that $f\in {\mathcal S}_{_{\scriptstyle  \mathbf p,\,\mu}}$, $\alpha, \beta>0$, $x,h\in {\mathbb R}$. Then

{\rm (i)} $\|\Delta_h^\alpha f\|_{_{\scriptstyle \mathbf p,\,\mu}}\le K(\alpha)\|f\|_{_{\scriptstyle \mathbf p,\,\mu}}$,  where \ \ $K(\alpha):=\sum_{j=0}^\infty \Big|{\alpha \choose j}\Big|\le 2^{\{\alpha\}}$,

~~~~$\{\alpha\}:=\inf\{k\in {\mathbb N}: k\ge \alpha\}$.

{\rm (ii)} ${[\Delta_h^\alpha f]}\widehat {\ \ }(k)=(1-\mathrm{e}^{-\mathrm{i}kh})^\alpha \widehat{f}(k)$, $ k\in\mathbb Z$.

{\rm (iii)} $(\Delta_h^\alpha (\Delta_h^\beta f))(x)=\Delta_h^{\alpha+\beta} f(x)$ (a.\,e.).

{\rm (iv)} $\|\Delta_h^{\alpha+\beta} f\|_{_{\scriptstyle \mathbf p,\,\mu}}\le 2^{\{\beta\}} \|\Delta_h^{\alpha} f\|_{_{\scriptstyle \mathbf p,\,\mu}} $.

{\rm (v)} $\lim\limits_{h\to 0}\|\Delta_h^{\alpha} f\|_{_{\scriptstyle \mathbf p,\,\mu}}=0$.
\end{lemma}

The proof of Lemma \ref{Lemma_1} and other auxiliary statements of the paper will be given in Section~\ref{Auxiliary statements}.


Based on 
(\ref{S_M.6}), the modulus of smoothness of $f\in {\mathcal S}_{_{\scriptstyle  \mathbf p,\,\mu}}$ of the index $\alpha>0$
is defined by
\[
    \omega_\alpha(f,\delta)_{_{\scriptstyle \mathbf p,\,\mu}}:=
    \sup\limits_{|h|\le \delta}\|\Delta_h^\alpha f\|_{_{\scriptstyle \mathbf p,\,\mu}}.
\]
Using the standard arguments, it can be shown that the functions
$\omega_\alpha(f,\delta)_{_{\scriptstyle \mathbf p,\,\mu}}$
possess all the basic properties of ordinary moduli of smoothness. Before formulating them,
we give the definition of the $\psi$-derivative of a function.


Let  $\psi=\{\psi_k\}_{k=-\infty}^\infty$ be an arbitrary sequence of complex numbers, $\psi_k\not=0$, $k\in {\mathbb Z}$. If for a given function $f\in L$ with the Fourier series of the form
$S[f](x)=\sum_{k\in {\mathbb Z}}\widehat {f}(k)\mathrm{e}^{\mathrm{i}k x},$ the series
$\sum_{k\in {\mathbb Z}\setminus\{0\}}\widehat {f}(k)\mathrm{e}^{\mathrm{i}k x}/{\psi_k} $
 is the  Fourier series of a certain function $g\in L$, then  $g$ is called
(see, for example, \cite[Ch.~9]{Stepanets_M2005}) $\psi$-derivative of the function $f$ and is denoted as $g:=f^{\psi}$. It is clear that the Fourier coefficients of functions  $f$ and  $f^{\psi }$  are related by equality
\begin{equation}\label{Fourier coeff}
    \widehat  f(k)=\psi_k\widehat  f^{\psi }(k), \ \ k\in {\mathbb Z}\setminus\{0\}.
\end{equation}
In the case $\psi_k=|k|^{-r}$, $r>0$, $k\in {\mathbb Z}\setminus\{0\}$, we use the notation $f^{\psi}=:f^{(r)}$.


\begin{lemma}\label{Lemma_2}  Assume that  $f, g\in {\mathcal S}_{_{\scriptstyle  \mathbf p,\,\mu}}$, $\alpha\ge \beta>0 $ and   $\delta,\delta_1,\delta_2>0$.  Then

{\rm (i)} $\omega_\alpha(f,\delta)_{_{\scriptstyle \mathbf p,\,\mu}}$
is a non-negative increasing continuous  function of $\delta$ on  $(0,\infty)$

~~~~such that $\lim\limits_{\delta\to 0+}\! \omega_\alpha(f,\delta)_{_{\scriptstyle \mathbf p,\,\mu}}\!\!=0$.

{\rm (ii)} $\omega_\alpha(f,\delta)_{_{\scriptstyle \mathbf p,\,\mu}}\le 2^{\{\alpha-\beta\}} \omega_\beta(f,\delta)_{_{\scriptstyle \mathbf p,\,\mu}}$.

{\rm (iii)} $\omega_\alpha(f+g,\delta)_{_{\scriptstyle \mathbf p,\,\mu}}\le \omega_\alpha(f,\delta)_{_{\scriptstyle \mathbf p,\,\mu}}+\omega_\alpha(g,\delta)_{_{\scriptstyle \mathbf p,\,\mu}}$.

{\rm (iv)}  $\omega _1(f,\delta_1+\delta_2)_{_{\scriptstyle \mathbf p,\,\mu}}\le
 \omega _1(f,\delta_1)_{_{\scriptstyle \mathbf p,\,\mu}}+\omega _1(f,\delta_2)_{_{\scriptstyle \mathbf p,\,\mu}}$.

 {\rm (v)}  $\omega _\alpha(f,\delta)_{_{\scriptstyle \mathbf p,\,\mu}}\le 2^{\{\alpha\}}\|f\|_{_{\scriptstyle \mathbf p,\,\mu}}$.

  {\rm (vi) }  if there exists 
  $f^{(\beta)} \in {\mathcal S}_{_{\scriptstyle  \mathbf p,\,\mu}}$, then
  $\omega _{\alpha}(f,\delta)_{_{\scriptstyle \mathbf p,\,\mu}} \le  \delta^{\beta}\omega _{\alpha-\beta}(f^{(\beta)},\delta)_{_{\scriptstyle \mathbf p,\,\mu}}$.

  {\rm (vii) }
\ $\omega_\alpha(f, p\delta)_{_{\scriptstyle \mathbf p,\,\mu}} \le p^\alpha \omega_\alpha(f, \delta)_{_{\scriptstyle \mathbf p,\,\mu}}$ \quad $({\alpha\in \mathbb{N}},\ \ {p \in \mathbb{N}}).$

{\rm (viii) }
\ $\omega_\alpha(f, \eta )_{_{\scriptstyle \mathbf p,\,\mu}} \le \delta^{-\alpha} (\delta+\eta)^\alpha  \omega_\alpha(f, \delta)_{_{\scriptstyle \mathbf p,\,\mu}}$ \quad
$({\alpha\in \mathbb{N}})$.

\end{lemma}



\section{ Direct approximation theorem.}

\begin{proposition}\label{Proposition 1} Let $\psi=\{\psi_k\}_{k=-\infty}^\infty$ be an arbitrary
sequence of complex numbers {such that}
$\psi_k\not=0$ and $\lim\limits_{|k|\to\infty}|\psi_k|=0$.
If for a function $f\in {\mathcal S}_{_{\scriptstyle  \mathbf p,\,\mu}}$  there exists a derivative
$f^{(\psi)}\in {\mathcal S}_{_{\scriptstyle  \mathbf p,\,\mu}}$, then the following inequality holds:
$$
    E_n(f)_{_{\scriptstyle \mathbf p,\,\mu}} \le \varepsilon_n E_n(f^{\psi})_{_{\scriptstyle \mathbf p,\,\mu}},\quad where\quad  \varepsilon_n=\max\limits_{|k|\ge n} |\psi_k|.
$$
\end{proposition}

\begin{proof}
According to  (\ref{S_M.4}) and  (\ref{Fourier coeff}), we have
  \begin{eqnarray}\nonumber
E_n (f)_{_{\scriptstyle \mathbf p,\,\mu}}
 &=& \inf\bigg\{a>0: \sum\limits_{|k|\ge n}\mu_k \Big({|\psi_k\widehat  f^{\psi }(k)|}/{a}\Big)^{p_k}\le 1\bigg\}
\\  \nonumber
&\le&\inf\bigg\{a>0: \sum\limits_{|k|\ge n}\mu_k \Big({\varepsilon_n|\widehat  f^{\psi }(k)|}/{a}\Big)^{p_k}\le 1\bigg\}\le \varepsilon_n E_n(f^{\psi})_{_{\scriptstyle \mathbf p,\,\mu}}.
\end{eqnarray}
\end{proof}

Note that if  $\varepsilon_n=\max\limits_{|k|\ge n} |\psi_k|=|\psi_{k_0}|$, where $k_0$ is   an integer, $|k_0|\ge n$, then for an arbitrary polynomial 
$\tilde{\tau}_{k_0}(x):=c\,\mathrm{e}^{\mathrm{i}k_0x}$, $c\not =0$, obviously, the equality holds:
$$
    E_n(\tilde{\tau}_{k_0})_{_{\scriptstyle \mathbf p,\,\mu}}= \varepsilon_n
    E_n(\tilde{\tau}_{k_0}^{\psi})_{_{\scriptstyle \mathbf p,\,\mu}}.
$$

\begin{theorem}\label{Theorem_1}  Assume that \ ${\bf p}=\{p_k\}_{k=-\infty}^\infty$  and ${\bf \mu}=\{\mu_k\}_{k=-\infty}^\infty$ are sequences of nonnegative numbers such that ${1<p_k \le K}$,  $k\in {\mathbb Z}$, and the function  $ f\in {\mathcal S}_{_{\scriptstyle  \bf p,\,\mu}}$. Then for any numbers $\alpha>0$ and
 $n \in \mathbb{N}$,  the following inequality holds:
$$
    E_n (f)_{_{\scriptstyle \bf p,\,\mu}} \le C(\alpha)\, \omega _\alpha(f; n^{-1})_{_{\scriptstyle \bf p,\,\mu}}.
$$
where $C=C(\alpha)$ is a constant that does not depend on $f$ and $n.$
\end{theorem}

Let us use the proof scheme from  \cite{Stechkin_1951}, where the similar estimates were obtained in the spaces $C^r({\mathbb T})$. In order to adapt this scheme in accordance with the properties of the spaces ${\mathcal S}_{_{\scriptstyle  \bf p,\,\mu}}$, before proving, we formulate the auxiliary Lemma \ref{Lemma_3}. This assertion
establishes the equivalence of the Luxembourg norm (\ref{S_M.1}) and the Orlicz norm, where the latter is defined as follows.

For given sequences ${\bf p}=\{p_k\}_{k=-\infty}^\infty$  and ${\bf \mu}=\{\mu_k\}_{k=-\infty}^\infty$ of nonnegative numbers such that ${1<p_k \le K}$, $k\in \mathbb{Z}$, consider the sequence ${\bf q}=\{q_k\}_{k\in \mathbb{Z}}$ defined by the equalities ${1/{p_k}+1/q_k=1}$, $k\in {\mathbb Z}$, 
and the set $\Lambda=\Lambda({\bf p},\mu)$ of all  numerical sequences  $\lambda=\{\lambda_k\}_{k\in \mathbb{Z}}$ such that  $\sum_{k\in \mathbb{Z}}\mu_k|\lambda_k|^{q_k}{\le} 1$. For any function  $f\in {\mathcal S}_{_{\scriptstyle  \bf p,\,\mu}}$, define its Orlicz norm by the equality
\begin{equation} \label{def-Orlicz-norm}
    \|f\|^\ast_{_{\scriptstyle \bf p,\,\mu}}:= \sup \Big\{ \sum\limits_{k \in \mathbb{Z}}
    \mu_k\lambda_k|\widehat{f}(k) |: \quad  \lambda\in \Lambda\Big\}.
\end{equation}

\begin{lemma}\label{Lemma_3} Assume that \ ${\bf p}=\{p_k\}_{k=-\infty}^\infty$  and ${\bf \mu}=\{\mu_k\}_{k=-\infty}^\infty$ are sequences of nonnegative numbers such that ${1<p_k \le K}$,  $k\in {\mathbb Z}$. Then for any function $f \in {\mathcal S}_{_{\scriptstyle  \bf p,\,\mu}}$,
\begin{equation} \label{estim-for-norms}
    \| f\|_{_{\scriptstyle \bf p,\,\mu}} \le \| f\|^\ast_{_{\scriptstyle \bf p,\,\mu}}\le 2 \,\| f\|_{_{\scriptstyle \bf p,\,\mu}}.
\end{equation}
\end{lemma}


\textit{Proof of Theorem \ref{Theorem_1}.} Let $\{K_n(t)\}_{n=1}^\infty$ be a sequence of kernels (where $K_n(t)$ is a trigonometric polynomial of order not greater than $n$), satisfying for all $n=1,2,\ldots$ the conditions:
\begin{equation} \label{K-cond-1}
    \int\limits_{-\pi}^\pi K_n(t)~{\rm d}t=1, 
\end{equation}
\begin{equation} \label{K-cond-3}
    \int\limits_{-\pi}^\pi |t|^r |K_n(t)|~{\rm d}t\le C(r) (n+1)^{-r} , \quad r=0,1,2,\ldots
\end{equation}
In the role of such kernels, in particular, we can take  the well-known Jackson kernels of sufficiently great order, that is,
$$
    K_n(t)=b_p\Big(\frac{\sin pt/2}{\sin t/2}\Big)^{2k_0},
$$
where $k_0$ is an integer that does not depend on  $n,~ 2k_0\ge r+2,$ the positive integer $p$ is
determined from the inequality
${n}/{(2k_0)}<p\le {n}/{(2k_0)}+1, $
and the constant $b_p$ is chosen due to the normalization condition (\ref{K-cond-1}).

It was shown in \cite{Stechkin_1951} that for any sequence of kernels  $\{K_n(t)\}$  satisfying conditions (\ref{K-cond-1})--(\ref{K-cond-3}), the following estimate holds:
\begin{equation} \label{K-cond-4}
    \int\limits_{-\pi}^\pi (|t|+n^{-1})^r ~|K_n(t)|~{\rm d}t\le C(r) n^{-r} , \quad (r, n=1,2,\ldots).
\end{equation}

Let us first consider the case of $\alpha \in \mathbb{N}$. Set
$$
    \sigma_{n-1}(x)= (-1)^{\alpha+1} \int\limits_{-\pi}^\pi K_{n-1}(t) \sum\limits_{j=1}^\alpha (-1)^j {\alpha \choose j} f({ x}-jt)~{\rm d}t.
$$
It is clear that $\sigma_{n-1}(x)$ is a trigonometric polynomial
which order does not exceed $n$. Further, in view of  (\ref{K-cond-1}), we have
  \begin{eqnarray}\nonumber
f(x)-\sigma_{n-1}(x)&=&    (-1)^{\alpha}\int\limits_{-\pi}^\pi K_{n-1}(t)
    \sum\limits_{j=0}^\alpha (-1)^j {\alpha \choose j} f({ x}-jt)~{\rm d}t\\ \nonumber
 &=& (-1)^{\alpha}\int\limits_{-\pi}^\pi K_{n-1}(t) \Delta_t^\alpha f(x)~{\rm d}t.
\end{eqnarray}

Hence, taking into account  relations (\ref{def-Orlicz-norm})--(\ref{estim-for-norms}) and the definition of the set $\Lambda$, we obtain
$$
    E_n (f)_{_{\scriptstyle \bf p,\,\mu}} \le \| f-\sigma_{n-1} \|_{_{\scriptstyle \bf p,\,\mu}} \le
    \| f-\sigma_{n-1} \|^\ast_{_{\scriptstyle \bf p,\,\mu}}=
     \Big\|(-1)^{\alpha}\int\limits_{-\pi}^\pi K_{n-1}(t) \Delta_t^\alpha f ~{\rm d}t  \Big\|^\ast_{_{\scriptstyle \bf p,\,\mu}}
$$
$$
    = \sup\Big\{\sum\limits_{k \in \mathbb{Z}}\mu_k \lambda_k\Big|\frac{1}{2\pi}
    \int\limits_{-\pi} ^\pi  \bigg(\int\limits_{-\pi} ^\pi
    K_{n-1}(t) \Delta_t^\alpha f(x)~ {\rm d}t\bigg)~\mathrm{e}^{-\mathrm{i}kx}~{\rm d}x  \Big|: ~
    \lambda \in \Lambda \Big\}.
$$
Applying now the Fubini theorem and again using estimate (\ref{estim-for-norms}), we find
  \begin{eqnarray}\nonumber
E_n (f)_{_{\scriptstyle \bf p,\,\mu}} \!\!\!\!&\le&\!\!\!\!     \int\limits_{-\pi} ^\pi |K_{n-1}(t)|
    \sup\Big\{\sum\limits_{k \in \mathbb{Z}}\mu_k
    \lambda_k\Big| \frac{1}{2\pi} \int\limits_{-\pi} ^\pi
    \Delta_t^\alpha f(x) \mathrm{e}^{-\mathrm{i}kx} {\rm d}x \Big| :
    \lambda\in \Lambda \Big\}{\rm d}t\\ \nonumber
\! \!\!\!&\le&\!\!\!\!  2 \int\limits_{-\pi} ^\pi|K_{n-1}(t)| \, \|\Delta_t^\alpha f(x)\|_{_{\scriptstyle \bf p,\,\mu}}^\ast  {\rm d}t \le
    2 \int\limits_{-\pi} ^\pi|K_{n-1}(t)|\,  \|\Delta_t^\alpha f(x)\|_{_{\scriptstyle \bf p,\,\mu}}  {\rm d}t
    \\ \label{int-K-omega}
     \!\!\!\!&\le&\!\!\!\!  2\int\limits_{-\pi} ^\pi |K_{n-1}(t)| \omega_\alpha(f;|t|)_{_{\scriptstyle \bf p,\,\mu}}~{\rm d}t.
\end{eqnarray}
To estimate the integral on the right-hand side of relation (\ref{int-K-omega}), we use the property (viii) of Lemma~\ref{Lemma_2}. Setting $\eta=|t|$, $\delta=n^{-1}$, we see that $\omega_\alpha(f; |t|)_{_{\scriptstyle \bf p,\,\mu}} \le n^\alpha (|t|+n^{-1})^\alpha \omega_\alpha(f; n^{-1})_{_{\scriptstyle \bf p,\,\mu}}$. Using this inequality and  (\ref{K-cond-4}), we get
\begin{eqnarray}\nonumber
\int\limits_{-\pi} ^\pi |K_{n-1}(t)| \omega_\alpha(f;|t|)_{_{\scriptstyle \bf p,\,\mu}} {\rm d}t &\!\!\le\!\!&  n^\alpha  \omega_\alpha(f; n^{-1})_{_{\scriptstyle \bf p,\,\mu}} \int\limits_{-\pi} ^\pi (|t|+n^{-1})^\alpha~|K_{n-1}(t)| {\rm d}t\\ \nonumber
 &\!\!\le\!\!&  C(\alpha) \omega_\alpha (f; n^{-1})_{_{\scriptstyle \bf p,\,\mu}}.
\end{eqnarray}
Thus, in the case of $\alpha \in \mathbb{N}$, the theorem is proved.

If $\alpha >0$, $\alpha \not \in \mathbb{N},$ then we denote by $\beta$ an arbitrary positive integer satisfying the condition $\beta-1<\alpha<\beta$. Due to property (ii) of Lemma \ref{Lemma_2}, we obtain
$$
    E_n (f)_{_{\scriptstyle \bf p,\,\mu}} \le C(\beta)~ \omega_\beta (f; n^{-1})_{_{\scriptstyle \bf p,\,\mu}} \le
    C(\beta)~ \omega_\alpha (f; n^{-1})_{_{\scriptstyle \bf p,\,\mu}}.
$$
\vskip -3mm$\hfill\Box$


\section{ Inverse approximation theorem.} Before proving the inverse approximation theorem, let us formulate the known Bernstein inequality in which the norm of the derivative of a trigonometric polynomial is estimated in terms of the norm of this polynomial (see, e.g.  \cite[Ch.~4]{A_Timan_M1960}), \cite[Ch.~4]{M_Timan_M2009}).

\begin{proposition}\label{Proposition 2} Let $\psi=\{\psi_k\}_{k=-\infty}^\infty$ be an arbitrary sequence of complex numbers, $\psi_k \not=0$. Then for any $\tau_n\in {\mathcal T}_{n}$, $n\in \mathbb{N}$, the following inequality holds:
$$
    \|\tau^\psi_n \|_{_{\scriptstyle \bf p,\,\mu}}\le \frac 1{\epsilon_n}\|\tau_n\|_{_{\scriptstyle \bf p,\,\mu}}, \quad
    \epsilon_n:=\min_{0<| k | \le n}|\psi_k|,
$$
\end{proposition}

\begin{proof}
Let $\tau_n(x)=\sum_{|k|\le n}  c_{k}\mathrm{e}^{\mathrm{i}({k,x})}$, $c_k\in {\mathbb C}$.  By the definition of the $\psi$-derivative and equalities (\ref{Fourier coeff}), we get
$$
    \|\tau^\psi_n \|_{_{\scriptstyle \bf p,\,\mu}}= 
       \inf \Big\{a>0:  \sum\limits_{0<| k | \le n}\mu_k  \Big({|c_k| /
       |a\psi_k |}\Big)^{p_k}\le 1\Big\}
$$
$$
    \le  \max_{0<| k | \le n}{|\psi_k|^{-1}} \inf \Big\{a>0: ~ \sum\limits_{0<| k | \le n}\mu_k
      \Big({| c_k| /  a}\Big)^{p_k}\le 1 \Big\}= \frac 1{\epsilon_n} \|\tau_n\|_{_{\scriptstyle \bf p,\,\mu}}.
$$
\end{proof}

Note that if  $\min\limits_{0<| k | \le n} |\psi_k|=|\psi_{k_0}|$,   then for an arbitrary polynomial of the form $\tilde{\tau}_{k_0}(x):=c\,\mathrm{e}^{\mathrm{i}k_0x}$, $c\not =0$, we have
\begin{eqnarray}\nonumber
\|\tilde{\tau}_{k_0}^\psi \|_{_{\scriptstyle \bf p,\,\mu}}&=&  \inf \Big\{a>0: \ \mu_{k_0}\Big({|c_{k_0}|}/{|a\psi_{k_0}|}\Big)^{p_k}\le 1 \Big\}\\ \nonumber
 &=&  \frac 1{|\psi_{k_0}|} \inf \Big\{a>0: \      \mu_{k_0}\Big({| c_{k_0}| /  a}\Big)^{p_k}\le 1 \Big\}=
    \frac 1{\epsilon_n} \| \tau_{k_0} \|_{_{\scriptstyle \bf p,\,\mu}}.
\end{eqnarray}


\begin{corollary}\label{Corollary 1} Let $\psi=\{\psi_k\}_{k=-\infty}^\infty$ be an arbitrary sequence of complex numbers {such that}  $|\psi_{-k}|=|\psi_k|\ge|\psi_{k+1}|>0$.  Then for any $\tau_n\in {\mathcal T}_{n}$, $n\in \mathbb{N}$,
$$
    \|\tau^\psi_n \|_{_{\scriptstyle \bf p,\,\mu}}\le \frac 1{|\psi_n|}\|\tau_n\|_{_{\scriptstyle \bf p,\,\mu}}.
$$
In particular, if $\psi_k=|k|^{-r}$, $r>0$, $k\in {\mathbb Z}\setminus\{0\}$, then
$$
    \|\tau^\psi_n \|_{_{\scriptstyle \bf p,\,\mu}} = \|\tau^{(r)}_n \|_{_{\scriptstyle \bf p,\,\mu}} \le
    n^r \|\tau_n\|_{_{\scriptstyle \bf p,\,\mu}}.
$$
\end{corollary}

\begin{theorem}\label{Theorem_2} If $ f\in {\mathcal S}_{_{\scriptstyle  \bf p,\,\mu}}$, then for any $\alpha>0$ and $n\in {\mathbb N}$, the following inequality is true:
\begin{equation}\label{S_M.12}
    \omega _\alpha\Big(f, \frac{\pi}{n}\Big)_{_{\scriptstyle \bf p,\,\mu}}\le \Big({\frac {\pi }n}\Big)^{\alpha}  \sum _{\nu =1}^{n}(\nu ^{\alpha}-(\nu -1)^{\alpha}) E_{\nu} (f)_{_{\scriptstyle \mathbf p,\,\mu}}.
\end{equation}
\end{theorem}

\begin{proof}

 Let us use the proof scheme from \cite{Stepanets_Serdyuk_2002}, modifying it taking into account the peculiarities of the spaces ${\mathcal S}_{_{\scriptstyle \bf p,\,\mu}}$. Let $ f\in {\mathcal S}_{_{\scriptstyle  \bf p,\,\mu}}$,  $n\in {\mathbb N}$ and $f_{jh}({x}):=f({x}-jh)$, where $j=0,1,\ldots$ and $h\in {\mathbb R}$. Then for any  ${ k}\in {\mathbb Z}$, we have  $ \widehat{f}_{jh}({ k}){=\widehat{f}({ k})\mathrm{e}^{-\mathrm{i}kjh}}$,
\begin{eqnarray}\nonumber
{[\Delta_h^\alpha f]}\widehat {\ \ }(k)&=&  \Big[\sum\limits_{j=0}^\infty (-1)^j
    {\alpha \choose j} f_{jh}\Big]\widehat{\ \ }(k)\\ \label{difference_Fourier_Coeff}
 &=&  \widehat{f}({ k})\sum\limits_{j=0}^\infty
    (-1)^j {\alpha \choose j}\mathrm{e}^{-\mathrm{i}kjh}=(1-\mathrm{e}^{-\mathrm{i}kh})^\alpha \widehat{f}(k).
\end{eqnarray}
and
\begin{equation}\label{modulus_difference_Fourier_Coeff}
|{[\Delta_h^\alpha f]}\widehat {\ \ }(k)|=|1-\mathrm{e}^{-\mathrm{i}kh}|^\alpha |\widehat{f}(k)|=2^\alpha
\Big|\sin \frac{kh}2\Big|^\alpha |\widehat{f}(k)|.
\end{equation}

Since $ f\in {\mathcal S}_{_{\scriptstyle  \bf p,\,\mu}}$, then for any $\varepsilon>0$ there exist a number $N_0=N_0(\varepsilon)\in {\mathbb N}$, $N_0>n$, such that for any $N>  N_0$, we have
 \[
  E_N(f)_{_{\scriptstyle \bf p,\,\mu}}=\|f-{S}_{N-1}({f})\|_{_{\scriptstyle \mathbf p,\,\mu}}
  <2^{-\alpha}\varepsilon.
 \]
Let us set $f_{0}:=S_{N_0}(f)$. Then in view of  (\ref{modulus_difference_Fourier_Coeff}), we see that
\begin{eqnarray}\nonumber
\|\Delta_h^\alpha f\|_{_{\scriptstyle \mathbf p,\,\mu}}&\le &  \|\Delta_h^\alpha f_{0}\|_{_{\scriptstyle \mathbf p,\,\mu}}+\|\Delta_h^\alpha (f-f_{0})\|_{_{\scriptstyle \mathbf p,\,\mu}}\le  \|\Delta_h^\alpha f_0\|_{_{\scriptstyle \mathbf p,\,\mu}} \\ \label{(6.3100)} &+&
 2^\alpha E_{N_0+1}(f)_{_{\scriptstyle \mathbf p,\,\mu}}<\|\Delta_h^\alpha f_0\|_{_{\scriptstyle \mathbf p,\,\mu}}+\varepsilon
\end{eqnarray}
Further, let ${S}_{n-1}:={S}_{n-1}(f_0)$  be the Fourier sum of $f_0$.  Then  by virtue of (\ref{modulus_difference_Fourier_Coeff}), for $|h|\le \pi /n$, we have
\begin{eqnarray}\nonumber
\|\Delta_h^\alpha f_0\|_{_{\scriptstyle \mathbf p,\,\mu}}&=&  \|\Delta_h^\alpha (f_0-S_{n-1})+\Delta_h^\alpha S_{n-1}\|_{_{\scriptstyle \mathbf p,\,\mu}}\\ \nonumber &\le&  \Big\| 2^\alpha (f_0-S_{n-1})+\sum _{|k|\le n-1}|kh|^\alpha |\widehat{f}(k)| {\mathrm{e}^{\mathrm{i}k\cdot}}\Big\|_{_{\scriptstyle \mathbf p,\,\mu}} \\ \nonumber &\le&
\Big\|2^\alpha (f_0-S_{n-1})+\Big({\frac {\pi }n}\Big)^{\alpha} \sum _{|k|\le n-1}|k|^\alpha |\widehat{f}(k)| {\mathrm{e}^{\mathrm{i}k\cdot}}\Big\|_{_{\scriptstyle \mathbf p,\,\mu}} \\ \label{(6.73)}
 &=&   \Big\|2^\alpha \sum _{\nu =n}^{N_0} H_{\nu}(\cdot)+\Big({\frac {\pi }n}\Big)^{\alpha}
     \sum _{\nu =1}^{n-1}\nu ^{\alpha}H_{\nu}(\cdot) \Big\|_{_{\scriptstyle \mathbf p,\,\mu}}.
\end{eqnarray}
where $H_{\nu}(x) :=H_{\nu}(f,x)=|\widehat{f}(\nu)| {\mathrm{e}^{\mathrm{i}\nu x}}+
|\widehat{f}(-\nu)| {\mathrm{e}^{-\mathrm{i}\nu x}}$, $\nu=1,2,\ldots$

Now we use the following assertion which is proved directly.

\begin{lemma}\label{Lemma_31} Let  $\{c_{\nu}\}_{\nu=1}^\infty$ and $\{a_{\nu}\}_{\nu=1}^\infty$ be arbitrary numerical sequences. Then the following equality holds for  all natural $m$, $M$ and $N$ $m\le M<N$:
\begin{equation}\label{(6.74)}
 \sum _{\nu =m}^Ma_{\nu }c_{\nu }=a_m\sum _{\nu=m}^{N }c_{\nu }+\sum _{\nu =m+1}^M(a_{\nu } -a _{\nu-1})\sum _{i=\nu }^{N }c_i-a_M\sum _{\nu =M+1}^{N }c_{\nu}.
    \end{equation}
\end{lemma}
Setting  $a_{\nu }=\nu^{\alpha},$  $c_{\nu }=H_{\nu}(x), $ $m=1$,  $M=n-1$ and $N=N_0$ in (\ref{(6.74)}), we get
 \[
  \sum _{\nu =1}^{n-1}\nu ^{\alpha}H_{\nu}(x)=
   \sum _{\nu =1}^{N_0}H_{\nu}(x)    +
   \sum _{\nu =2}^{n-1}(\nu ^{\alpha}-(\nu -1)^{\alpha})
    \sum_{i=\nu }^{N_0}H_{i}(x) -(n-1)^{\alpha}\sum
_{\nu =n}^{N_0 }H_{\nu}(x).
  \]
Therefore,
  \[
 \Big\|2^\alpha \sum _{\nu =n}^{N_0} H_{\nu}(\cdot)+\Big({\frac {\pi }n}\Big)^{\alpha}
     \sum _{\nu =1}^{n-1}\nu ^{\alpha}H_{\nu}(\cdot) \Big\|_{_{\scriptstyle \mathbf p,\,\mu}}\le
   \Big({\frac {\pi }n}\Big)^{\alpha}
    \Big\|n^\alpha\sum _{\nu =n}^{N_0} H_{\nu}(\cdot)
     \]
     \begin{eqnarray}\nonumber
&+&\sum _{\nu =1}^{n-1}(\nu ^{\alpha}-(\nu -1)^{\alpha})
    \sum_{i=\nu }^{N_0}H_{i}(\cdot) -(n-1)^{\alpha}\sum
_{\nu =n}^{N_0}H_{\nu}(\cdot)  \Big\|_{_{\scriptstyle \mathbf p,\,\mu}} \\ \nonumber
 &\le& \Big({\frac {\pi }n}\Big)^{\alpha}
    \Big\|\sum _{\nu =1}^{n}(\nu ^{\alpha}-(\nu -1)^{\alpha})
    \sum_{i=\nu }^{N_0}H_{i}(\cdot) \Big\|_{_{\scriptstyle \mathbf p,\,\mu}}\\ \label{(6.74aqq)}
 &\le&   \Big({\frac {\pi }n}\Big)^{\alpha}  \sum _{\nu =1}^{n}(\nu ^{\alpha}-(\nu -1)^{\alpha}) E_{\nu} (f_0)_{_{\scriptstyle \mathbf p,\,\mu}}.
\end{eqnarray}
Combining relations (\ref{(6.3100)}), (\ref{(6.73)})  and (\ref{(6.74aqq)}) and taking into account the definition of the function $f_0$, we see that for  $|h|\le \pi /n$, the following inequality holds:
\[
  \|\Delta_h^\alpha f\|_{_{\scriptstyle \mathbf p,\,\mu}}\le
  \Big({\frac {\pi }n}\Big)^{\alpha}  \sum _{\nu =1}^{n}(\nu ^{\alpha}-(\nu -1)^{\alpha}) E_{\nu} (f)_{_{\scriptstyle \mathbf p,\,\mu}}+\varepsilon
\]
 which, in view of arbitrariness of $\varepsilon$,  gives us (\ref{S_M.12}).
\end{proof}

In the spaces ${\mathcal S}^p$, similar results were obtained in  \cite{Sterlin_1972} and \cite{Stepanets_Serdyuk_2002}. In the  Orlicz type spaces ${\mathcal S}_M$ of functions $f\in L$ with the finite norm
\[
 \|{f}\|_{_{\scriptstyle  M}}:=\|\{\widehat{f}(k)\}_{k\in {\mathbb Z}}\|_{_{\scriptstyle l_{M}({\mathbb Z})}}=
    \inf\Big\{a{>}0:\  \sum_{k\in\mathbb Z}  M(|{\widehat{f}(k)}|/{a})\le 1\Big\},
\]
     where $M$ is an Orlicz function,  direct and inverse theorems were proved in \cite{Chaichenko_Shidlich_Abdullayev}. Unlike the results of \cite{Chaichenko_Shidlich_Abdullayev}, here we also get the  constant $\pi^\alpha$ in inequality (\ref{S_M.12}). This constant is exact   in the sense that for any positive number $\varepsilon>0$, there exists a function $f^*\in {\mathcal S}_{_{\scriptstyle  \bf p,\,\mu}}$ such that for all $n$  greater that a certain number $n_0$, we have
 \begin{equation}\label{(6.31abcd)}
  \omega_{\alpha}^{\ }\Big(f^*, \frac {\pi }n\Big)_{_{\scriptstyle \bf p,\,\mu}}>
 \frac {\pi ^\alpha-\varepsilon }{n^\alpha }\sum _{\nu =1}^n(\nu ^{\alpha }-(\nu -1)^{\alpha })E_{\nu}(f^*)_{_{\scriptstyle \mathbf p,\,\mu}}.
 \end{equation}
 Consider the function  $f^*(x)={\rm e}^{\mathrm{i}k_0x}$, where $k_0$ is an arbitrary positive integer. Then   $E_{\nu}(f^*)_{_{\scriptstyle \bf p,\,\mu}}=1$ for $\nu=1,2,\ldots,k_0$,  $E_{\nu}(f^*)_{_{\scriptstyle \bf p,\,\mu}}=0$  for $\nu>k_0$ and
 \[
  \omega_{\alpha}^{\ }\Big(f^*, \frac {\pi }n\Big)_{_{\scriptstyle \bf p,\,\mu}}\ge
 \|\Delta _{\frac {\pi }n}^\alpha f^*\|_{_{\scriptstyle \bf p,\,\mu}} \ge
  2^{\alpha  }\Big|\sin \frac {k_0 \pi}{2n}\Big|^{\alpha   }.
 \]
 Since ${\sin t}/t$\ \  tends\ \  to $1$ as $t\to 0$, then for all $n$   greater that a certain number $n_0$, the  inequality   $2^{\alpha  } |\sin {k_0 \pi}/{(2n)}|^{\alpha  }>
 {(\pi^\alpha-\varepsilon)} k_0^\alpha/{n^\alpha}$ holds, which  yields  (\ref{(6.31abcd)}).

Since $\nu ^{\alpha}-(\nu -1)^{\alpha}\le \alpha  \nu ^{\alpha-1},$ it follows from inequality  (\ref{S_M.12}) that
 \begin{equation}\label{(6.71')}
 \omega_{\alpha}^{\ }\Big(f, \frac {\pi }n\Big)_{_{\scriptstyle \bf p,\,\mu}}\le
 \frac {\pi ^\alpha \alpha}{n^\alpha }\sum _{\nu =1}^n\nu ^{\alpha-1}E_{\nu}(f)
 _{_{\scriptstyle \bf p,\,\mu}}.
 \end{equation}
This, in particular, yields the following statement:

\begin{corollary}\label{Corollary 2} Assume that the sequence of the best approximations $E_n(f)_{_{\scriptstyle \bf p,\,\mu}}$ of a function $f\in {\mathcal S}_{_{\scriptstyle  \bf p,\,\mu}}$ satisfies the following relation for some $\beta >0$:
$$
    E_n(f)_{_{\scriptstyle \bf p,\,\mu}}= {\mathcal O}(n^{-\beta }).
$$
Then, for any $\alpha>0$, one has
$$
    \omega _\alpha(f, t)_{_{\scriptstyle \bf p,\,\mu}}=
    \left \{ \begin{matrix}  {\mathcal O}(t^{\beta }) & \hfill \mbox {for} \ \ \beta <\alpha, \hfill \cr
                        {\mathcal O}(t^\alpha|\ln t|) & \hfill \mbox {for}\ \ \beta =\alpha, \hfill \cr
                        {\mathcal O}(t^\alpha) & \hfill \mbox {for} \ \ \beta >\alpha.\hfill
                        \end{matrix} \right.
$$
\end{corollary}

For the spaces $L_p$ of $2\pi$-periodic functions integrable to the $p$th power with the usual norm, inequalities of the type (\ref{(6.71')}) were proved by M.~Timan (see, for example \cite[Ch.~6]{A_Timan_M1960}, \cite[Ch.~2]{M_Timan_M2009}).


\section{ Constructive characteristics of the classes of functions defined by the
$\alpha$th moduli of smoothness.}
In the following two sections some applications of the obtained results are considered.
In particular, in this section we give the constructive characteristics of the classes
${\mathcal S}_{_{\scriptstyle  \bf p,\,\mu}}H_{\omega_\alpha} $ of functions for which the $\alpha$th moduli of smoothness do not exceed some majorant.

Let $\omega$ be  a function defined on interval $[0,1]$. For a fixed $\alpha>0$, we set
\begin{equation} \label{omega-class}
    {\mathcal S}_{_{\scriptstyle  \bf p,\,\mu}}H^{\omega}_{\alpha}=
    \Big\{f\in {\mathcal S}_{_{\scriptstyle  \bf p,\,\mu}}:  \quad \omega_\alpha(f; \delta)_{_{\scriptstyle \bf p,\,\mu}}=
    {\mathcal O}  (\omega(\delta)),\quad  \delta\to 0+\Big\}.
\end{equation}
Further, we consider the functions $\omega(t)$, $t\in [0,1]$, satisfying the following conditions 1)--4):

\noindent {\bf 1)} $\omega(t)$ is continuous on $[0,1]$;\ \  {\bf 2)} $\omega(t)\uparrow$;\ \  {\bf 3)} $\omega(t)\not=0$ for any $t\in (0,1]$; \ \  {\bf 4)} $\omega(t)\to 0$ as $t\to 0$; as well-known condition $({\mathcal B}_\alpha)$, $\alpha>0$  (see, e.g. \cite{Bari_Stechkin_1956}):
$\displaystyle{\sum_{v=1}^n v^{\alpha-1}\omega(t^{-1}) =
{\mathcal O}  \Big[n^\alpha \omega (n^{-1})\Big]}$.

\begin{theorem}\label{Theorem_3}  Assume that \ ${\bf p}=\{p_k\}_{k=-\infty}^\infty$  and ${\bf \mu}=\{\mu_k\}_{k=-\infty}^\infty$ are sequences of nonnegative numbers such that ${1<p_k \le K}$, $k\in {\mathbb Z}$, $\alpha>0$  and $\omega$ is a function, satisfying  conditions  $1)$--\,$4)$ and $({\mathcal B}_\alpha)$. Then a function $f\in {\mathcal S}_{_{\scriptstyle  \bf p,\,\mu}}$ belongs to the class ${\mathcal S}_{_{\scriptstyle  \bf p,\,\mu}}H^{\omega}_{\alpha}$  if and only if
\begin{equation} \label{iff-theorem}
    E_n(f)_{_{\scriptstyle \bf p,\,\mu}}={\mathcal O} \Big[ \omega (n^{-1}) \Big].
\end{equation}
\end{theorem}

\begin{proof}
Let $f \in {\mathcal S}_{_{\scriptstyle  \bf p,\,\mu}}H^{\omega}_{\alpha}$, by virtue of
Theorem \ref{Theorem_1}, we have
$$
    E_n(f)_{_{\scriptstyle \bf p,\,\mu}} \le C(\alpha)  \omega _\alpha (f, n^{-1})_{_{\scriptstyle \bf p,\,\mu}}.
$$
Therefore, {relation (\ref{omega-class}) yields (\ref{iff-theorem})}.
On the other hand, if relation (\ref{iff-theorem}) holds, then by virtue of (\ref{(6.71')}),  taking into account the condition $({\mathcal B}_\alpha)$, we obtain
     \begin{eqnarray}\nonumber
\omega _\alpha(f, n^{-1})_{_{\scriptstyle \bf p,\,\mu}}&\le&
\frac {C(\alpha) }{n^\alpha} \sum _{\nu =1}^n \nu ^{\alpha-1} E_{\nu}(f)_{_{\scriptstyle \bf p,\,\mu}}
 \\ \nonumber  &\le& \frac {C_1}{n^\alpha} \sum _{\nu =1}^n \nu ^{\alpha-1} \omega (\nu^{-1})=
    {\mathcal O}  \Big[\omega ( n^{-1})\Big].
\end{eqnarray}
Thus, the function  $f$ belongs to the set ${\mathcal S}_{_{\scriptstyle  \bf p,\,\mu}}H^{\omega}_{\alpha}$. \end{proof}

The function $\varphi (t)=t^r$, $r \le \alpha$, satisfies the condition $({\mathcal B}_\alpha)$. Hence, denoting by ${\mathcal S}_{_{\scriptstyle  \bf p,\,\mu}}H_{\alpha}^r$  the class ${\mathcal S}_{_{\scriptstyle  \bf p,\,\mu}}H^{\omega}_{\alpha}$ for $\omega(t)=t^r$, $0<r\le \alpha,$  we establish the following statement:

\begin{corollary}\label{Corollary 3} Assume that \ ${\bf p}=\{p_k\}_{k=-\infty}^\infty$  and ${\bf \mu}=\{\mu_k\}_{k=-\infty}^\infty$ are sequences of nonnegative numbers such that ${1<p_k \le K}$, $k\in {\mathbb Z}$, and  $\alpha >0$, $0<r\le \alpha.$  Then a function $f\in {\mathcal S}_{_{\scriptstyle  \bf p,\,\mu}}$ belongs to the class ${\mathcal S}_{_{\scriptstyle  \bf p,\,\mu}}H_{\alpha}^r$  if and only if
$$
    E_n(f)_{_{\scriptstyle \bf p,\,\mu}}={\mathcal O}   ({n^{-r}} ).
$$
\end{corollary}


\section{ The equivalence between $\alpha$th moduli of smoothness and $K$-functionals.}  $K$-functionals were introduced by Lions and Peetre in 1961, and defined in their usual form by Peetre in 1963. Unlike the moduli of continuity expressing the smooth properties of functions, $K$-functionals express some of their approximative properties. In this section, we prove the equivalence of our moduli of smoothness and certain Peetre $K$-functionals. This connection is important for studying the properties of the modulus of smoothness and the $K$-functional, and also for their further application to the problems of approximation theory.

In the space ${\mathcal S}_{_{\scriptstyle  \bf p,\,\mu}}$, the Petree $K$-functional of a  function $f$
(see, {e.g.} \cite[Ch.~6]{DeVore_Lorentz_M1993}), {which} generated by its derivative of order $\alpha>0$, is the following quantity:
$$
    K_\alpha(\delta,f)_{_{\scriptstyle \bf p,\,\mu}}:=\inf\Big\{\|f-h\|_{_{\scriptstyle \bf p,\,\mu}}+
    \delta^\alpha \|h^{(\alpha)}\|_{_{\scriptstyle \bf p,\,\mu}}:\
    h^{(\alpha)}\in {\mathcal S}_{_{\scriptstyle  \bf p,\,\mu}}\Big\},\quad \delta>0.
$$

\begin{theorem}\label{Theorem_4} Assume that \ ${\bf p}=\{p_k\}_{k=-\infty}^\infty$  and ${\bf \mu}=\{\mu_k\}_{k=-\infty}^\infty$ are sequences of nonnegative numbers such that ${1<p_k \le K}$, $k\in {\mathbb Z}$. Then for each $ f\in {\mathcal S}_{_{\scriptstyle  \bf p,\,\mu}}$ and $\alpha>0$, there exist constants $C_1(\alpha)$, $C_2(\alpha)>0$, such that for $\delta>0$
\begin{equation}\label{KO1}
      C_1(\alpha)\omega _\alpha(f, \delta)_{_{\scriptstyle \bf p,\,\mu}}\le
      K_\alpha(\delta,f)_{_{\scriptstyle \bf p,\,\mu}}\le
      C_2(\alpha)\omega _\alpha(f, \delta)_{_{\scriptstyle \bf p,\,\mu}}.
\end{equation}
\end{theorem}

Before proving Theorem \ref{Theorem_4},  let us formulate the following auxiliary  Lemma \ref{Lemma_4}, which is used to prove the right-hand side of  (\ref{KO1}).

\begin{lemma}\label{Lemma_4} Assume that  $\alpha >0$, $n \in \mathbb{N}$ and $0<h< 2\pi/n$.
 Then for any  $\tau_n \in {\mathcal T}_n $
\begin{equation}\label{Bermstain-inequl-gener}
    \Big( \frac{\sin nh/2}{n/2}\Big)^\alpha\| \tau_n^{(\alpha)} \|_{_{\scriptstyle \bf p,\,\mu}}\le
     \|\Delta_h^\alpha \tau_n \|_{_{\scriptstyle \bf p,\,\mu}} \le
     h^\alpha \| \tau_n^{(\alpha)} \|_{_{\scriptstyle \bf p,\,\mu}}.
\end{equation}
\end{lemma}

\textit{Proof of Theorem \ref{Theorem_4}.} Consider an arbitrary function $h\in {\mathcal S}_{_{\scriptstyle  \bf p,\,\mu}}$ such that
$h^{(\alpha)}\in {\mathcal S}_{_{\scriptstyle  \bf p,\,\mu}}$. Then  we have by
Lemma \ref{Lemma_2} (iii), (v) and (vi)
$$
    \omega _\alpha(f, \delta)_{_{\scriptstyle \bf p,\,\mu}}\le
    \omega _\alpha(f-h, \delta)_{_{\scriptstyle \bf p,\,\mu}}+
    \omega _\alpha(h, \delta)_{_{\scriptstyle \bf p,\,\mu}}\le 2^{\{\alpha\}}\|f-h\|_{_{\scriptstyle \bf p,\,\mu}}+ \delta^\alpha\|h^{(\alpha)}\|_{_{\scriptstyle \bf p,\,\mu}}.
$$
Taking the infimum over all  $h\in {\mathcal S}_{_{\scriptstyle  \bf p,\,\mu}}$ such that
$h^{(\alpha)}\in {\mathcal S}_{_{\scriptstyle  \bf p,\,\mu}}$,  we get the left-hand side of (\ref{KO1}).

Now let  $\delta \in (0,2\pi)$ and  $n \in \mathbb{N}$ such that
$\pi/n<\delta < 2\pi/n$. Let also $S_n:=S_n(f)$ be the Fourier sum of $f$.    Using Lemma \ref{Lemma_4} with
$h=\pi /n$ and  property  (i) of Lemma \ref{Lemma_1}, we obtain
\[
\|S_n^{(\alpha)}\|_{_{\scriptstyle \bf p,\,\mu}}\!\le\!
 (n/2)^{\alpha} \|\Delta_{\pi/n}^\alpha S_n \|_{_{\scriptstyle \bf p,\,\mu}}
  \le  (\pi/\delta)^{\alpha} \Big( \|\Delta_{\pi/n}^\alpha (S_n -f)  \|_{_{\scriptstyle \bf p,\,\mu}}+ \|\Delta_{\pi/n}^\alpha f  \|_{_{\scriptstyle \bf p,\,\mu}}\Big)
 \]
 \begin{equation} \label{KO2}
 \le  (\pi/\delta)^{\alpha} \Big( 2^{\{\alpha\}} \| f- S_n \|_{_{\scriptstyle \bf p,\,\mu}}+
    \|\Delta_{\pi/n}^\alpha f  \|_{_{\scriptstyle \bf p,\,\mu}}\Big).
\end{equation}
By virtue of (\ref{S_M.4}) and Theorem \ref{Theorem_1}, we have
\begin{equation}\label{KO3}
    \|f- S_n \|_{_{\scriptstyle \bf p,\,\mu}}=E_n(f)_{_{\scriptstyle \bf p,\,\mu}} \le
    C(\alpha) \omega_\alpha (f; \delta )_{_{\scriptstyle \bf p,\,\mu}}.
\end{equation}
Combining (\ref{KO2}), (\ref{KO3})    and the definition of modulus of smoothness, we obtain the relation
$$
    \|S_n^{(\alpha)} \|_{_{\scriptstyle \bf p,\,\mu}} \le
    C_2(\alpha) \delta^{-\alpha} \omega_\alpha (f; \delta )_{_{\scriptstyle \bf p,\,\mu}},
$$
where $C_2(\alpha):=\pi^\alpha(2^{\{\alpha\}}C(\alpha)+1)$, which yields the right-hand side of (\ref{KO1}):
$$
    K_\alpha(\delta,f)_{_{\scriptstyle \bf p,\,\mu}} \le \|f- S_n \|_{_{\scriptstyle \bf p,\,\mu}}+
    \delta^\alpha \|S_n^{(\alpha)}\|_{_{\scriptstyle \bf p,\,\mu}} \le
    C_2(\alpha) \omega_\alpha (f, \delta )_{_{\scriptstyle \bf p,\,\mu}}.
$$
\vskip -3mm$\hfill\Box$


\section{Proof of auxiliary statements}\label{Auxiliary statements}

{\it 8.1 Proof of Lemma \ref{Lemma_1}.}  By virtue of (\ref{difference_Fourier_Coeff}), we have
\[
    \|\Delta_h^\alpha f\|_{_{\scriptstyle \mathbf p,\,\mu}}=\inf\bigg\{a>0: \sum_{{ k}\in {\mathbb Z}} \mu_k\Big|\widehat{f}({ k})\sum\limits_{j=0}^\infty (-1)^j {\alpha \choose j}\mathrm{e}^{-\mathrm{i}kjh}/a\Big|^{p_k}\le 1\bigg\}.
\]
where for any $a>0$, the following inequalities hold:
     \begin{eqnarray}\nonumber
\sum_{{ k}\in {\mathbb Z}} \mu_k\Big|\widehat{f}({ k})\sum\limits_{j=0}^\infty
    (-1)^j {\alpha \choose j}\mathrm{e}^{-\mathrm{i}kjh}/a\Big|^{p_k}&\le&
\sum_{{ k}\in {\mathbb Z}}  \mu_k\Big(\sum\limits_{j=0}^\infty \Big| {\alpha \choose j}\Big|
    |\widehat{f}({ k})|/a\Big)^{p_k}
 \\ \nonumber  &\le& \sum_{{ k}\in {\mathbb Z}}
    \mu_k\Big( 2^{\{\alpha\}}  |\widehat{f}({ k})|/a\Big)^{p_k},
\end{eqnarray}
and hence property {\rm (i)} is true. Property {\rm (ii)} follows from (\ref{difference_Fourier_Coeff}) and property (iii) is its consequence.  Part {\rm (iv)} follows
  from
{\rm (i)}--{\rm (iii)}.


To prove {\rm (v)} we first show that the following relation holds:
\begin{equation}\label{Trigon_Polynomial}
    \lim\limits_{|h|\to 0}\|\Delta_h^{\alpha} \tau_{n}\|_{_{\scriptstyle \mathbf p,\,\mu}}=0
\end{equation}
where $\tau_{n}$ is a 
polynomial of the form $\tau_{n}(x){=}\sum_{|k|\le n}  c_{k}\mathrm{e}^{\mathrm{i}kx}$, $n\in {\mathbb N}$, $c_{k}\in {\mathbb C}$.

Since $\|\tau_{n}\|_{_{\scriptstyle \mathbf p,\,\mu}}=\inf\{a>0: \sum_{|k|\le n}\mu_k( |c_{k}|/a)^{p_k}\le 1\}$, then by virtue of (\ref{modulus_difference_Fourier_Coeff}), for $a_0=|nh|^\alpha \|\tau_{n}\|_{_{\scriptstyle \mathbf p,\,\mu}}$, we obtain
\[
 \sum_{|k|\le n} \mu_k\Big(|{[\Delta_h^\alpha \tau_{n}]}\widehat {\ \ }(k)|/a_0\Big)^{p_k}=
    \sum_{|k|\le n} \mu_k\Big(|1-\mathrm{e}^{-{\mathrm i}kh}|^\alpha  {|c_{k}|}/{a_0}\Big)^{p_k}
    \]
 \begin{eqnarray}\nonumber
&=&
\sum_{|k|\le n} \mu_k\Big(2^\alpha \Big|\sin ({kh}/2)\Big|^\alpha {|c_{k}|}/{a_0}\Big)^{p_k}
\le    \sum_{|k|\le n} \mu_k\Big( |kh|^\alpha  {|c_{k}|}/{a_0}\Big)^{p_k}
 \\ \label{Trigon_P}  &\le& \sum_{|k|\le n} \mu_k\Big( |nh|^\alpha \frac{|c_{k}|}{a_0}\Big)^{p_k}=
     \sum_{|k|\le n} \mu_k\Big( |c_{k}|/\|\tau_{n}\|_{_{\scriptstyle \mathbf p,\,\mu}}\Big)^{p_k}\le 1.
\end{eqnarray}
Therefore, $\|\Delta_h^{\alpha} \tau_{n}\|_{_{\scriptstyle \mathbf p,\,\mu}}\le |nh|^\alpha \|\tau_{n}\|_{_{\scriptstyle \mathbf p,\,\mu}}$. For an arbitrary $\varepsilon>0$,  we set $\delta:=\delta(\varepsilon)=\Big(\varepsilon/n^\alpha\|\tau_{n}\|_{_{\scriptstyle \mathbf p,\,\mu}}\Big)^{1/\alpha}$. Then for all $|h|<\delta$, we have
$\|\Delta_h^{\alpha} \tau_{n}\|_{_{\scriptstyle \mathbf p,\,\mu}}<\varepsilon$, i.e., relation (\ref{Trigon_Polynomial}) is indeed fulfilled.

Now let $f$ be a  function from ${\mathcal S}_{M}$ and $S_n=S_{n}(f)$ its Fourier sum.  Then for any  $\varepsilon>0$ there exist a number $n_0=n_0(\varepsilon)$ such that
for any  $n>n_0$, we have $\|f-{S}_{n}\|_{_{\scriptstyle \mathbf p,\,\mu}}<{\varepsilon}/ 2^{\{\alpha\}+1}$.
Furthermore, by virtue of  (\ref{Trigon_Polynomial}),   there exist a number  $\delta:=\delta(\varepsilon,n)$ such that $\|\Delta_h^{\alpha} S_{n} \|_{_{\scriptstyle \mathbf p,\,\mu}}<\frac {\varepsilon}2$ when $|h|<\delta$. Then using properties of norm and {\rm (i)}, for  $n>n_0$ we get the following relation which   yields {\rm (v)}:
\[
\|\Delta_h^\alpha f\|_{_{\scriptstyle \mathbf p,\,\mu}}\le
\|\Delta_h^\alpha (f-S_n)\|_{_{\scriptstyle \mathbf p,\,\mu}}+
    \|\Delta_h^\alpha S_n\|_{_{\scriptstyle \mathbf p,\,\mu}}
 \!\!\le 2^{\{\alpha\}}\| f-S_n\|_{_{\scriptstyle \mathbf p,\,\mu}}+
    \|\Delta_h^\alpha S_n\|_{_{\scriptstyle \mathbf p,\,\mu}}\!\!<\varepsilon.
\]
$\hfill\Box$

{\it 8.2. Proof of Lemma \ref{Lemma_2}.}
Property {\rm (iii)}, non-negativity and increasing of the function $\omega_{\alpha}(f,t)_{_{\scriptstyle \mathbf p,\,\mu}}$  follow from the definition of modulus of smoothness. In {\rm (i)}, the convergence to zero for $\delta\to 0+$ follows by
{\rm (v)} of  Lemma \ref{Lemma_1}.  Property  {\rm (v)} is the consequence of Lemma \ref{Lemma_1}  {\rm (i)}.  According to {\rm (i)} and {\rm (iii)} of Lemma~\ref{Lemma_1}, for arbitrary
$0<\alpha\le \beta$, we have
${\|\Delta_h^\alpha f \|_{_{\scriptstyle \mathbf p,\,\mu}}=
\|\Delta_h^{\alpha-\beta} (\Delta_h^\beta f) \|_{_{\scriptstyle \mathbf p,\,\mu}} \le    2^{\alpha-\beta}
\|\Delta_h^\beta f \|_{_{\scriptstyle \mathbf p,\,\mu}}},$
whence passing to the exact upper bound over all $|h|\le \delta$, we obtain {\rm (ii)}.
Property {\rm (iv)} is proved by  the 
usual arguments. In particular, this property  yields the continuity of the function $\omega_1(f,\delta)_{_{\scriptstyle \mathbf p,\,\mu}}$, since for arbitrary $\delta_1>\delta_2>0$,
$\omega _1(f,\delta_1)_{_{\scriptstyle \mathbf p,\,\mu}}-\omega _1(f,\delta_2)_{_{\scriptstyle \mathbf p,\,\mu}}\le
\omega_1(\delta_1-\delta_2)_{_{\scriptstyle \mathbf p,\,\mu}}\to 0$ as $\delta_1-\delta_2\to 0.$

Let us prove the continuity of the function $\omega_\alpha(f,\delta)_{_{\scriptstyle \mathbf p,\,\mu}}$ for  arbitrary $\alpha>0$. Let ${0<\delta_1<\delta_2}$ and $h=h_1+h_2,$ where $0<h_1\le \delta_1$, $0<h_2\le \delta_2-\delta_1.$ Since
$$    \Delta_h^\alpha f(\delta) =\Delta_{h_1}^\alpha f (\delta) +\sum_{j=0}^\infty {\alpha \choose j}
    (-1)^{j} \Delta_{jh_2}^1 f(\delta+jh_1)
$$
and
$$
    \Big\|\sum_{j=0}^\infty {\alpha \choose j}
    (-1)^{j} \Delta_{jh_2}^1 \, f_{jh_1}\Big\|_{_{\scriptstyle \mathbf p,\,\mu}}
$$
\begin{eqnarray}\nonumber
 &=&
    \inf\bigg\{a>0: \sum_{{ k}\in {\mathbb Z}}\mu_k \Big|{\Big[\sum_{j=0}^\infty {\alpha \choose j}
    (-1)^{j} \Delta_{jh_2}^1  f_{jh_1}\Big]}\widehat {\ \ }(k)/a\Big|^{p_k}\le 1\bigg\}
 \\ \nonumber  &\le& \inf\bigg\{a>0: \sum_{{ k}\in {\mathbb Z}}\mu_k
    \Big( 2^{\{\alpha\}}\alpha  |[\Delta_{h_2}^1 f  ]\widehat {\ \ }(k) |/a\Big)^{p_k}\le 1\bigg\}\le
    2^{\{\alpha\}}\alpha\|\Delta_{h_2}^1 f\|_{_{\scriptstyle \mathbf p,\,\mu}},
\end{eqnarray}
then
$
    \|\Delta_h^\alpha f\|_{_{\scriptstyle \mathbf p,\,\mu}}  \le
    \|\Delta_{h_1}^\alpha f\|_{_{\scriptstyle \mathbf p,\,\mu}} +
    2^{\{\alpha\}}\alpha\|\Delta_{h_2}^1 f\|_{_{\scriptstyle \mathbf p,\,\mu}}
$
and
$$
    \omega_\alpha (f,\delta_2)_{_{\scriptstyle \mathbf p,\,\mu}}   \le
    \omega_\alpha (f,\delta_1)_{_{\scriptstyle \mathbf p,\,\mu}} +
    2^{\{\alpha\}}\alpha \,\omega_1 (f, \delta_2-\delta_1)_{_{\scriptstyle \mathbf p,\,\mu}}.
$$
Hence, we obtain the necessary relation:
$$
    \omega_\alpha(f,\delta_2)_{_{\scriptstyle \mathbf p,\,\mu}}
    - \omega_\alpha(f,\delta_1)_{_{\scriptstyle \mathbf p,\,\mu}}\le 2^{\{\alpha\}}\alpha \,
    \omega_1(f, \delta_2-\delta_1)_{_{\scriptstyle \mathbf p,\,\mu}} \to 0, \quad \delta_2-\delta_1\to 0.
$$

If there exists a derivative $f^{(\beta)}\in {\mathcal S}_{_{\scriptstyle  \mathbf p,\,\mu}}$, $0<\beta\le \alpha$, then by virtue of (\ref{difference_Fourier_Coeff}) and (\ref{Fourier coeff}), for arbitrary numbers
$k\in {\mathbb Z}\setminus\{0\}$ and $h\in [0,\delta]$, we have
\begin{eqnarray}\nonumber
 |{[\Delta_h^\alpha f]}\widehat {\ \ }(k)| &=&
    2^\beta |\sin ({kh}/)2|^\beta |1-\mathrm{e}^{-\mathrm{i}kh}|^{\alpha-\beta} |\widehat{f}(k)|
 \\ \nonumber  &\le& \delta^\beta|k|^\beta   |1-\mathrm{e}^{-\mathrm{i}kh}|^{\alpha-\beta}|\widehat{f}(k)|\le \delta^\beta  |{[\Delta_h^{\alpha-\beta} f^{(\beta)}]}\widehat {\ \ }(k)|,
\end{eqnarray}
and therefore property (vi) holds.

If ${\alpha}$ and $p$ are positive integers,  then using the representation
$$
    \Delta_{ph}^\alpha f(x)= \sum_{k_1=0}^{p-1} \ldots \sum_{k_\alpha=0}^{p-1}
    \Delta_h^\alpha f(x-(k_1+k_2+\ldots+k_\alpha) h),
$$
and the relation
$$
    \Big| [\Delta_h^\alpha f(x-(k_1+k_2+\ldots+k_\alpha) h)]\widehat { \ \ } (k) \Big|
$$
\begin{eqnarray}\nonumber
 &=&
    \Big|\frac{1}{2\pi} \int\limits_{-\pi}^\pi \sum\limits_{j=0}^\alpha (-1)^{j} {\alpha \choose j} f_{jh}(x-(k_1+k_2+\ldots+k_\alpha) h) \mathrm{e}^{-\mathrm{i}kx}~{\rm d}x\Big|
 \\ \nonumber  &\le& \Big|\frac{1}{2\pi} \int\limits_{-\pi}^\pi \sum\limits_{j=0}^\alpha (-1)^{j} {\alpha \choose j} f_{jh}(x) \mathrm{e}^{-\mathrm{i}kx}~{\rm d}x\Big|=\Big| [\Delta_h^\alpha f(x)]\widehat { \ \ } (k) \Big|,
\end{eqnarray}
we see that $\| \Delta_{ph}^\alpha f(x) \|_{_{\scriptstyle \mathbf p,\,\mu}}\le p^\alpha \| \Delta_{h}^\alpha f(x) \|_{_{\scriptstyle \mathbf p,\,\mu}}$:
$$
    \inf\Big\{a>0:
    \sum_{k\in {\mathbb Z}}\mu_k \Big(  \Big| \sum_{k_1=0}^{p-1}   \ldots \sum_{k_\alpha=0}^{p-1} [\Delta_h^\alpha f(x -(k_1+ \ldots+k_\alpha) h)]\widehat { \ \ } (k)  \Big| /a   \Big)^{p_k}\le 1 \Big\}
    $$
$$
  \le \inf\Big\{a>0: \sum_{k\in {\mathbb Z}}\mu_k
    \Big( p^\alpha \Big| [\Delta_h^\alpha f(x)]\widehat { \ \ } (k)  \Big| /a   \Big)^{p_k} \le 1 \Big\}
    .
$$
 To prove (viii) it is sufficient to consider the case $\delta< \eta$
(for $\delta \ge \eta$,  property (viii) is obvious). Choosing the number $p$ such that ${\eta/
\delta} \le p < {\eta/
\delta}+1$, by virtue (i) and (vii), we obtain
$$
    \omega_\alpha(f; \eta)_{_{\scriptstyle \mathbf p,\,\mu}} \le \omega_\alpha(f; p \delta)_{_{\scriptstyle \mathbf p,\,\mu}} \le p^\alpha \omega_\alpha(f;  \delta)_{_{\scriptstyle \mathbf p,\,\mu}} \le
    ({\eta/\delta}+1)^{\alpha}   \omega_\alpha(f, \delta)_{_{\scriptstyle \mathbf p,\,\mu}}.\eqno\Box
$$

\subsection{\it Proof of Lemma \ref{Lemma_3}.}
The right-hand side of  (\ref{estim-for-norms})
is obtained from the Young inequality
$$
    a b \le \frac{a^{p}}{p}+\frac{b^{q}}{q}, \quad \frac{1}{p}+\frac{1}{q}=1, \quad
    a\ge0,\quad b\ge0,
$$
 as follows (here 
 in the proof, we exclude the trivial case when $f\equiv {\rm const}$)
$$
    \| {f\|^\ast_{_{\scriptstyle \bf p,\,\mu}}/\|f\|_{_{\scriptstyle \bf p,\,\mu}}} =
    \Big\| {f/\|f\|_{_{\scriptstyle \bf p,\,\mu}}} \Big\|^\ast_{_{\scriptstyle \bf p,\,\mu}}=
    \sup \Big\{ \sum_{k \in \mathbb{Z}} {  \mu_k|\lambda_k\widehat{f}(k)}|/
    {\|f\|_{_{\scriptstyle \bf p,\,\mu}}} : \  \lambda\in \Lambda\Big\}
$$
$$
    \le \sup \Big\{\sum\limits_{k \in \mathbb{Z}}
    \mu_k\Big(   \frac{|\widehat{f}(k)/\|f\|_{_{\scriptstyle \bf p,\,\mu}}|^{p_k}}{p_k}+
     \frac{|\lambda_k|^{q_k}}{q_k}\Big): \ \lambda\in \Lambda\Big\}
$$
$$
     \le \sup \Big\{\sum\limits_{k \in \mathbb{Z}}
    \mu_k\Big(   \Big|\widehat{f}(k)/\|f\|_{_{\scriptstyle \bf p,\,\mu}}\Big|^{p_k}+
     |\lambda_k|^{q_k}\Big): \ \lambda\in \Lambda\Big\} \le 2.
$$

To prove the left-hand side of 
(\ref{estim-for-norms}), let us show that for any function
${f \in {\mathcal S}_{_{\scriptstyle  \bf p,\,\mu}}}$, from the inequality $\|f\|^\ast_{_{\scriptstyle \bf p,\,\mu}} \le 1$, it follows that $\sum_{k \in \mathbb{Z}}\mu_k |\widehat{f}(k)|^{p_k}\le 1$. Indeed, assume that $\sum_{k \in \mathbb{Z}} \mu_k|\widehat{f}(k)|^{p_k}>1$. Then take a number $\rho > 1$  such that $ \sum_{k \in \mathbb{Z}}\mu_k |{\widehat{f}(k)/ \rho} |^{p_k}=1$ and consider the sequence $\tilde{\lambda}=\{\tilde{\lambda}_k\}_{k\in \mathbb{Z}}$ defined by the equalities
$\tilde{\lambda}_k= (|\widehat{f}(k)|/\rho)^{p_k-1}$
for $k\in {\mathbb Z}$. We have
$$
    \sum\limits_{k \in \mathbb{Z}}\mu_k |\tilde{\lambda}_k|^{q_k}=
    \sum\limits_{k \in \mathbb{Z}}\mu_k \Big|{\widehat{f}(k)/ \rho}\Big|^{(p_k-1)q_k}=
    \sum\limits_{k \in \mathbb{Z}}\mu_k \Big|{\widehat{f}(k)/ \rho} \Big|^{p_k}= 1
$$
that is, $\tilde{\lambda}\in \Lambda({\bf p},\mu)$.  However, by the  definition (\ref{def-Orlicz-norm}) of the Orlicz norm, we get
$$
    \|f\|^\ast_{_{\scriptstyle \bf p,\,\mu}} \ge \sum\limits_{k \in \mathbb{Z}}\mu_k\tilde{\lambda}_k |\widehat{f}(k)|=  \rho \sum\limits_{k \in \mathbb{Z}}\mu_k \Big|{\widehat{f}(k)/
    \rho}\Big|^{p_k}=\rho>1,
$$
which is a contradiction. Hence, for any function $f \in {\mathcal S}_{_{\scriptstyle  \bf p,\,\mu}}$,
the inequality $\|f\|^\ast_{_{\scriptstyle \bf p,\,\mu}} \le 1$ yields the inequality $\sum_{k \in \mathbb{Z}} \mu_k|\widehat{f}(k)|^{p_k}\le 1$.

Since
$
    \Big\| {f}/ {\|f\|^\ast_{_{\scriptstyle \bf p,\,\mu}}} \Big\|^\ast_{_{\scriptstyle \bf p,\,\mu}}=1,
$
then
$
     \sum\limits_{k \in \mathbb{Z}}\mu_k
    \Big|{\widehat{f}(k)}/{\|f\|^\ast_{_{\scriptstyle \bf p,\,\mu}}}\Big|^{p_k} \le 1
$
and therefore, $\|f\|_{_{\scriptstyle \bf p,\,\mu}}\le \|f\|^\ast_{_{\scriptstyle \bf p,\,\mu}}$. 
$\hfill\Box$

{\it Proof of Lemma \ref{Lemma_4}.}Since for any polynomial of the form $\tau_{n}(x){=}\sum_{|k|\le n}  c_{k}\mathrm{e}^{\mathrm{i}kx}$ we have
$\|\tau_{n}^{(\alpha)}\|_{_{\scriptstyle \bf p,\,\mu}}=\inf\{a>0: \sum_{|k|\le n}\mu_k(|k|^\alpha |c_{k}|/a)^{p_k}\le 1\}$, then similarly to (\ref{Trigon_P}), we obtain
\[
    \sum_{|k|\le n} \mu_k\Big(|{[\Delta_h^\alpha \tau_{n}]}\widehat {\ \ }(k)|/a_1\Big)^{p_k}\le
     \sum_{|k|\le n} \mu_k\Big( |kh|^\alpha |c_{k}|/a_1\Big)^{p_k}
     \]
     \[
     \le
     \sum_{|k|\le n} \mu_k
     \Big(|k|^\alpha |c_{k}|/\|\tau_{n}^{(\alpha)}\|_{_{\scriptstyle \bf p,\,\mu}}\Big)^{p_k}\le 1,
\]
when $a_1:=|h|^\alpha \|\tau_{n}^{(\alpha)}\|_{_{\scriptstyle \bf p,\,\mu}}$. Therefore,  $\|\Delta_h^{\alpha} \tau_{n}\|_{_{\scriptstyle \bf p,\,\mu}} \le |h|^\alpha \|\tau_{n}^{(\alpha)}\|_{_{\scriptstyle \bf p,\,\mu}}$.

In (\ref{Bermstain-inequl-gener}), the first inequality  is trivial in the cases where $h=0$ or $|h|=2\pi/n$. So, now let $0<|h|<2\pi/n$. Since
\[
    \|\Delta_h^{\alpha} \tau_{n}\|_{_{\scriptstyle \bf p,\,\mu}}=
    \inf\Big\{a>0: \sum_{|k|\le n} \mu_k
    \Big(2^\alpha \Big|\sin \frac {kh}2\Big|^\alpha |c_{k}|/a\Big)^{p_k}\le 1\Big\}
\]
and  the function $t/\sin t$ increase on $(0,\pi)$, then for  $a_2:=\Big|\frac{n/2}{\sin (nh/2)}\Big|^\alpha \|\Delta_h^{\alpha} \tau_{n}\|_{_{\scriptstyle \bf p,\,\mu}}$ we get
\begin{eqnarray}\nonumber
\sum_{|k|\le n} \mu_k (|k|^\alpha |c_{k}|/{a_2})^{p_k} &=&
    \sum_{|k|\le n} \mu_k \Big(\Big|\frac{kh/2}{\sin (kh/2)}\Big|^\alpha
    \Big|\frac{\sin (kh/2)}{h/2}\Big|^\alpha |c_{k}|/{a_2}\Big)^{p_k}
 \\ \nonumber  &\le&\sum_{|k|\le n} \mu_k \Big(\Big|\frac{nh/2}{\sin (nh/2)}\Big|^\alpha
    \Big|\frac{\sin (kh/2)}{h/2}\Big|^\alpha |c_{k}|/{a_2}\Big)^{p_k}
 \\ \nonumber  &=& \sum_{|k|\le n} \mu_k \Big(2^\alpha \Big|\sin \frac {kh}2\Big|^\alpha |c_{k}|/\|\Delta_h^{\alpha} \tau_{n}\|_{_{\scriptstyle \bf p,\,\mu}}\Big)^{p_k}\le 1.
\end{eqnarray}
Thus, the first inequality in (\ref{Bermstain-inequl-gener}) also holds. $\hfill\Box$

\end{document}